\documentclass{amsart}

\usepackage{amssymb}
\usepackage{amsmath}
\usepackage{amsthm}
\usepackage{amsbsy}
\usepackage{bm}
\usepackage{hyperref}
\date{\today}
\usepackage{cite}
\usepackage{array}

\theoremstyle{theorem}
    \newtheorem{theorem}{Theorem}
    \newtheorem{lemma}{Lemma}

\theoremstyle{definition} 
    
    \newtheorem{fact}{Fact}
    \newtheorem{result}{Result}
    \newtheorem{remark}{Remark}
    \newtheorem{example}[theorem]{Example}
    \newtheorem{exercise}[theorem]{Exercise}
    
    \newtheorem{assumption}{Assumption}
    
\def\d{\delta}
\def\e{\epsilon}

\def\vp{\varphi}

\def\C{\mathbb{C}}

\def\N{\mathbb{N}}

\def\R{\mathbb{R}}

\def\l{\left}
\def\r{\right}
\def\<{\langle}
\def\>{\rangle}

\newcommand{\E}{\mbox{\bf E}}
\def\I#1{{\bf 1}_{#1}}

\def\bar{\overline}

\newcommand\Tr{{\mbox{Tr}}}
\newcommand\tr{{\mbox{tr}}}

\newcommand\mnote[1]{} 
\newcommand\be{\begin{equation*}}

\newcommand\ee{\end{equation*}}

\newcommand\ben{\begin{equation}}
\newcommand\een{\end{equation}}
\newcommand\bes{\begin{eqnarray*}}
\newcommand\ees{\end{eqnarray*}}

\newcommand\bex{\begin{exercise}}
\newcommand\eex{\end{exercise}}
\newcommand\beg{\begin{example}}
\newcommand\eeg{\end{example}}
\newcommand\benu{\begin{enumerate}}
\newcommand\eenu{\end{enumerate}}
\newcommand\beit{\begin{itemize}}
\newcommand\eeit{\end{itemize}}
\newcommand\berk{\begin{remark}}
\newcommand\eerk{\end{remark}}
\newcommand\bdefn{\begin{defintion}}
\newcommand\edefn{\end{definition}}
\newcommand\bthm{\begin{theorem}}
\newcommand\ethm{\end{theorem}}
\newcommand\bprf{\begin{proof}}
\newcommand\eprf{\end{proof}}
\newcommand\blem{\begin{lemma}}
\newcommand\elem{\end{lemma}}

\newcommand{\sm}{{\raise0.3ex\hbox{$\scriptstyle \setminus$}}}

\def\l{\left}
\def\r{\right}

\def\CHI{\mathchoice%
{\raise2pt\hbox{$\chi$}}%
{\raise2pt\hbox{$\chi$}}%
{\raise1.3pt\hbox{$\scriptstyle\chi$}}%
{\raise0.8pt\hbox{$\scriptscriptstyle\chi$}}}
\def\smalloplus{\raise1pt\hbox{$\,\scriptstyle \oplus\;$}}

\title{Brown measure and asymptotic freeness of elliptic and related matrices}

\author{Kartick Adhikari}
\address{Stat-Math Unit\\
        Indian Statistical Institute\\
        Kolkata 700108, India}

\email{kartickmath [at] gmail.com}

\author{Arup Bose}
\address{Stat-Math Unit\\
        Indian Statistical Institute\\
        Kolkata 700108, India}
        
\email{bosearu [at] gmail.com}

\date{\today}
\thanks{The work is partially supported by National Post-Doctoral Fellowship, India, with reference no. {\bf PDF/2016/001601}, and also  supported by J. C. Bose National Fellowship, Department of Science and Technology, Government of India}
\begin{document}
\maketitle

\begin{abstract}
We show that independent elliptic matrices converge to freely independent elliptic elements. Moreover, the elliptic matrices are asymptotically free with deterministic matrices under appropriate conditions. We compute the Brown measure of the product of elliptic elements. It turns out that this Brown measure is same as the limiting spectral distribution.
\end{abstract}

\section{Introduction}

\textit{Asymptotic $*$-freeness} (in short freeness) of a sequence of random matrices, as the dimension increases, was introduced by Voiculescu \cite{voiculescu1}. It is a central object of study in free probability. 
Later it has been studied extensively in the literature by \cite{dykema}, \cite{hiai}, \cite{dimitri}, \cite{speicher},  \cite{voiculescu2} and others. In particular it is known that under suitable assumptions, independent standard unitary matrices and deterministic matrices are  asymptotically free, and so are independent Wigner and deterministic matrices. 
Further, let $Y_{p\times n}$ be a $p\times n$ random rectangular matrix where the entries  are i.i.d. Gaussian random variables with mean zero and variance one. It is known that (see Theorem 5.2 in \cite{capitaine}) independent copies of $\frac{1}{n}Y_{p\times n}Y_{n\times p}^\prime$ are also asymptotically free.

A generalisation of the Wigner matrix that has caught recent attention is the elliptic matrix where the entries are as in a Wigner matrix except that the $(i,j)$th and $(j,i)$th entries have a correlation which is same across all pairs. This is clearly a non-symmetric matrix. One may also allow some specific pattern of correlation instead of the constant correlation. We call the latter a generalised elliptic matrix. It is known that under suitable conditions the limit spectral distribution (LSD) of the elliptic matrix is the uniform distribution on an ellipse whose axes depend on the value of the correlation. See \cite{naumov}, \cite{sommer},  \cite{nguyen}.

We first show that under suitable conditions, the sequence of generalised elliptic matrices converges in $*$-distribution 
(see Theorem \ref{thm:general}). In particular any sequence of elliptic matrices converges to an elliptic element.
 
Next we show  that  under appropriate conditions, independent elliptic matrices, with possibly different correlation values, converge jointly in $*$-distribution and are asymptotically  free. They are also asymptotically free of appropriate collection of deterministic matrices. See Theorem \ref{thm2}.
The joint convergence should remain true for generalised elliptic matrices but freeness is not expected to remain valid. To keep things simple we decided not to pursue these ideas. 

Now consider the empirical spectral distribution (ESD) of $\frac{1}{n}Y_{p\times n}Y_{n\times p}^\prime$ when $p/n \to y\neq 0$. 
When the entries of $Y_{p\times n}$ are i.i.d. random variables with mean zero, variance one and all moments finite, this  converges  to the Mar\v cenko-Pastur law almost surely.  Other variations under weaker assumptions are also known. See \cite{baisilversteinbook}, \cite{pastur}, \cite{wachter} and \cite{yin}. Now suppose $Y_{p \times n}$ is elliptic. We show that then  the \textit{expected} ESD still converges to the Mar\v cenko-Pastur law. See Theorem \ref{thm:rectangle}. Again, we have not pursued the almost sure convergence of the ESD for simplicity. We also show that independent copies of $\frac{1}{n}Y_{p\times n}Y_{n\times p}^\prime$ converge jointly in $*$-distribution and are asymptotically free. See Theorem \ref{thm:assymrect}.

The Brown measure for any element of a non-commutative probability space was introduced by Brown \cite{brown}. There has been a lot of work done in the past decade to find connection between the limiting spectral distribution (LSD) of a sequence of random matrices and the Brown measure of the $*$-distribution limit of the sequence. Often they are not equal. A very simple example is given in \cite{sniady}. 

However, often they are equal. For example the i.i.d. matrix converges in $*$-distribution to the circular element and its LSD is the uniform distribution on the unit disc. The latter is indeed the Brown measure of the circular element. See \cite{haagerup}. 
The LSD of  the elliptic matrix is the uniform probability measure on  an ellipse. At the same time, the Brown measure of an elliptic element is also the uniform probability measure on an ellipse (see \cite{belinschi}, \cite{larsen}). Similalry, in  \cite{manjusinglering}, it has been shown that the LSD of  bi-unitarily random matrices is actually the Brown measure of the $*$-distribution limit. 

The LSD of product of elliptic matrices has been calculated  in \cite{soshnikov}. 
On the other hand, from Theorem \ref{thm2}, we know that this product converges in $*$-distribution to product of free elliptic elements. We calculate the Brown measure of such a product and show that it is the  same as the LSD. See Theorem \ref{thm:brown}.

We introduce the basic definitions and facts in Section \ref{sec:pre} and state our results in Section \ref{sec:results}. In Sections \ref{sec:limit} and \ref{sec:asypt} we give the proofs of Theorems  \ref{thm:general} and \ref{thm2} respectively.  Proof of Theorems \ref{thm:rectangle} and \ref{thm:assymrect} are given in Section \ref{sec:rect} and the proof of Theorem \ref{thm:brown} is presented in Section \ref{sec:brown}.

\section{Preliminaries}\label{sec:pre}
We 
first recall some basic definitions and facts from  free probability theory. 
 A self-adjoint element $s$, in a non-commutative probability space (NCP) $(\mathcal A,\vp)$, is said to be a (standard) { \it semi-circular element} if $$
\vp(s^p)=\frac{1}{2\pi}\int_{-2}^2t^p\sqrt {4-t^2}dt, \;\mbox{ for $p\in \N$}.
$$ 
These are moments of the probability density $\frac{1}{2\pi}\sqrt {4-t^2}$ on the interval $[-2,\ 2]$ (see \cite{speicherbook}, p. 29). 

Elements (called random variables) $a_1, a_2,\ldots, a_n\in \mathcal A$ in an NCP $(\mathcal A, \vp)$ are said to be {\it free} if 
for any polynomials $p_1,p_2,\ldots,p_k$ whenever $\varphi(p_j(a_{i(j)}))=0, 1\le j\le k$ and $i(1)\neq i(2)\neq \cdots \neq i(k)$, we have 
$$
\varphi(p_1(a_{i(1)})\cdots p_k(a_{i(k)}))=0.
$$

An element $c\in \mathcal{A}$ is called {\it circular} if $c=\frac{1}{\sqrt 2}s_1+\frac{i}{\sqrt 2}s_2$,  where $s_1$ and $s_2$ are two free semi-circular elements. An element $e\in \mathcal A$ is called {\it elliptic} with parameter $\rho$ if $e=\sqrt{\frac{1+\rho}{2}}s_1+i\sqrt{\frac{1-\rho}{2}}s_2$, where $s_1$ and $s_2$ are free semi-circular elements. Note that $\rho=1$ and $\rho=0$ yield respectively the semi-circular and the circular element.

Let $(\mathcal A_n,\vp_n)_{n\ge 1}$ be a sequence of NCP. Let $(a_i^n)_{i\in I}$ be a collection of random variables from $\mathcal A_n$ which converges in $*$-distribution to some $(a_i)_{i\in I}$ in  $(\mathcal A, \vp)$.
Then $(a_i^{(n)})_{i\in I}$ are said to be asymptotically free if $(a_i)_{i\in I}$ are free.

Let $\mathcal A_n$ be the algebra of  $n\times n $ random matrices whose entries have all moments finite. It is equipped with the tracial state $\varphi_n(x)=\frac{1}{n}\E\Tr(x)$ for $x\in \mathcal A_n$. Clearly a sequence of random matrices $(A_n)$ from  $\mathcal A_n$, {\it converges in $*$-distribution} to some element $a\in \mathcal A$ 
if for every choice of $\epsilon_1,\epsilon_2,\ldots, \epsilon_k\in \{1, *\}$ we have 
$$
\lim_{n\to \infty}\varphi_n(A_n^{\epsilon_1}\cdots A_n^{\epsilon_k})=\varphi(a^{\epsilon_1}\cdots a^{\epsilon_k}).
$$
Then we write $A_n\stackrel{*\mbox{-dist}}{\longrightarrow}a$. 
If $A_n$ is in addition hermitian, then the above condition is same as saying 
 $\lim \vp_n(a_n^k)=\vp(a^k)$ exists  for all non-negative integers $k$. 
Then we say $A_n$ converges to $a$ in the distribution sense. 
The joint convergence of several sequences is expressed in an analogous manner. 

There are other related notions of convergence of random matrices. Let $A_n$ be an $n\times n$ random matrix with eigenvalues $\lambda_1,\ldots, \lambda_n$. Then $\frac{1}{n}\sum_{k=1}^n\d_{\lambda_k}$, where $\d_x$ denotes the Dirac delta measure at $x$, is the {\it empirical spectral measure} of $A_n$. Equivalently, the {\it empirical spectral distribution} (ESD) is given by
\begin{align*}
F^{A_n}(x,y)=\frac{|\{k\;:\; \Re(\lambda_k)\le x, \Im(\lambda_k)\le y\}|}{n},\;\;\;\mbox{ for $x,y\in \R$},
\end{align*} 
where $|\cdot|$ denotes the cardinality and $\Im$ and $\Re$ denote the imaginary and real parts.
Clearly $F^{A_n}$ is a random distribution function. 
If, as $n \to \infty$,  it converges (almost surely) to a non-random distribution function $F_{\infty}$ weakly, then $F_{\infty}$ is said to be the almost sure {\it limiting spectral distribution} (LSD) of $A_n$. 
Often it is easier to show the convergence of the non-random  {\it expected} ESD function $\E[F^{A_n}(x,y)]$. This limit is also called LSD and coincides with the earlier limit if both exist. 
If the sequence of matrices is hermitian, then its convergence in the distribution sense yields the candidate LSD whose moments are $\vp(a^k)$.  

We now state a few well known facts, which will be used in the proofs of our results. We need the following notation:
\begin{align*}
\mathcal P_2(2k)&:=\mbox{the set of all pair partitions of $\{1,2,\ldots, 2k\}$,}\\
NC_2(2k)&:=\mbox{the set of non-crossing pair partitions of $\{1,2,\ldots, 2k\}$,}
\\ \gamma &:= \mbox{the cyclic permutation with one cycle, i.e.,}
\\& \;\;\;\;\; \hspace{.1cm}\mbox{$\gamma(i)=i+1$ for $i=1,\ldots, 2k-1$ and $\gamma(2k)=1$}
\\ |A|&= \mbox{the cardinality of $A$},
\\ \gamma\pi(r) &:= \gamma(\pi(r)) \mbox{ for } \pi\in \mathcal P_2(2k), \mbox{where $\pi(r)=s$ and $\pi(s)=r$ if $(r,s)\in \pi$,}
\\I_{k}&:=\{(i_1,\ldots,i_{k})\in \N^{k}\;  :\; 1\le i_1,\ldots, i_{k} \le n\},
\end{align*}

\begin{fact}[Moments-free cumulants] \label{ft:moment-cumulant}
Let $a_1,a_2,\ldots, a_n\in (\mathcal A, \vp)$. Then 
\begin{align*}
\vp(a_1a_2\cdots a_n)=\sum_{\pi\in NC(n)}\prod_{V\in \pi}\kappa(V)[a_1,a_2,\ldots,a_n],
\end{align*}
where $NC(n)$ denotes the set of all non-crossing partition of $\{1,\ldots,n\}$ and $\kappa (V)$ denotes the usual multiplicative extension of the free cumulant function. See \cite{speicherbook} for details on the definition of free cumulants and its multiplicative extension. 
\end{fact}

\begin{fact}[Wick's formula]\label{ft:wick}
Let $G_1,G_2,\ldots, G_{j}$ be jointly Gaussian random variables with $\E[G_i]=0$ for $i=1,2,\ldots, j$. Then 
$$
\E[G_1G_2\cdots G_{j}]=\sum_{\pi\in \mathcal P_2(j)} \prod_{(r,s)\in \pi}\E[G_rG_s].
$$
\end{fact}

\begin{fact}\label{ft:cardianlity}Let $\pi\in\mathcal P_2(2k)$ and  $|\gamma \pi|$ denote the cardinality of the set of partition blocks in $\gamma \pi$. Then
$|\gamma \pi|\le k+1$, and equality holds if and only if $\pi \in NC_2(2k)$.
\end{fact}

For the proofs of the above three Facts 
refer to \cite{speicherbook}, pages 176, 360 and  367 respectively. To illustrate Fact \ref{ft:cardianlity}, let $\{1,\ldots,6\}$ so that 
$k=3$. Let $\pi=(13)(24)(56)$ be a crossing pair partition. Then 
$\gamma\pi=(14325)(6)$, using the convention that if $(r,s)\in \pi$ then $r=\pi(s)$ and $s=\pi(r)$, so $|\gamma\pi|=2$. On the other hand, if $\pi=(12)(34)(56)$, which is non-crossing, then $\gamma\pi=(135)(2)(4)(6)$ and $|\gamma\pi|=4$. 
  
Now we introduce the Brown measure.  Let $(\mathcal A, \vp)$ be an NCP. Then the Fuglede-Kadison  determinant (see \cite{fuglede}) $\Delta (a) $ of  $a\in \mathcal A$ is defined  by
\begin{align*}
\Delta (a)= \exp[\frac{1}{2}\vp(\log (aa^*)) ].
\end{align*}	
If $a$ is not invertible, then $\Delta (a):=\lim_{\e\to 0}\Delta_{\e} (a)$, where $\Delta_{\e} (a)$ denotes the regularized Fuglede-Kadison determinant $\Delta_{\e} (a)=\exp[\frac{1}{2}\vp(\log (aa^*+\e^2)) ]$ for $\e>0$.

The Brown measure of $a\in \mathcal A$ is defined by \cite{brown}
$$
\mu_a=\frac{1}{2\pi}\left(\frac{\partial^2}{\partial(\Re \lambda)^2}+\frac{\partial^2}{\partial(\Im \lambda)^2}\r)\log\Delta (a-\lambda)=\frac{2}{\pi}\frac{\partial}{\partial\lambda}\frac{\partial}{\partial\bar \lambda}\log\Delta (a-\lambda).
$$
One can show that in fact $\mu_a$ is a probability measure on $\C$. Consider any $n\times n$ matrix $A_n$. Then 
\begin{align*}
\Delta(A_n)=\sqrt[n]{|\det A_n|}, \mbox{ and } \mu_{A_n}=\frac{1}{n}\sum_{i=1}^n\d_{\lambda_i}
\end{align*}
respectively, where $\lambda_1,\ldots, \lambda_n$ are the eigenvalues of $A_n$. So the Brown measure is the ESD of the matrix. See \cite{sniady} for details. 
However, even when the ESD converges, there is no guarantee that the limit is the Brown measure of the limit element.

\section{Main results}\label{sec:results}

The following assumption is basic to us.  

 \begin{assumption} \label{assu} The variables  
$\{a_{ii}: 1\le i\}\cup \{(a_{ij}, a_{ji}): 1\le i<j\}$ is a collection of independent random variables, and satisfy 
$\E[a_{ij}]=0$, $\E[a_{ij}^2]=1$ for all $i, j$, and have uniformly bounded moments of all orders. 
\begin{enumerate}
\item[{\bf (i)}] (Elliptic matrix) $\E[a_{ij}a_{ji}]=\rho$ for $1\le i\neq j\le n$.

\vspace{.2cm}
\item[{\bf (ii)}] (Generalised elliptic matrix) $\E[a_{ij}a_{ji}]=\rho_{|i-j|}$ for $1\le i\neq j\le n$.
\end{enumerate}
\end{assumption}
\noindent  Note that {\bf (i)} is a particular case of {\bf (ii)}. Any  rectangular matrix $X_{p\times n}$ is defined to be elliptic or generalised elliptic by adapting the above definition in the obvious way. 
In the literature an elliptic matrix is assumed to also satisfy the condition that $\{a_{ii}: 1\le i\}$ and $\{(a_{ij}, a_{ji}): 1\le i<j\}$ are collections of i.i.d. random variables. 

The first result is on the $*$-distribution limit for 
generalised elliptic matrices.

\begin{theorem}\label{thm:general}
Suppose $A_n$ is a sequence of generalised elliptic matrices whose entries satisfy Assumption \ref{assu}{\bf (ii)}.  Suppose further that $\{\rho_i\}$ satisfies: 
for all  $\e_1,\ldots, \e_{2k}\in \{1,*\}$, and for all $\pi\in NC_2(2k)$, 
\begin{align}\label{eqn:condition}
\lim_{n\to \infty}\frac{1}{n^{k+1}}\sum_{I_{2k}}\prod_{(r,s)\in \pi}\rho_{|i_r-i_s|}^{\d_{\e_r\e_s}}\d_{i_ri_{s+1}}\d_{i_{s}i_{r+1}}<\infty,\;\;.
\end{align}

Then $\bar A_n:=\frac{A_n}{\sqrt n}$ converges in $*$-distribution. 
Moreover, if $\rho_i\equiv \rho$ for all $i$, then (\ref{eqn:condition}) holds  
and the limit is elliptic with parameter $\rho$.
\end{theorem} 
Condition \eqref{eqn:condition} also holds if $\rho_{|i-j|}=f\l(\frac{|i-j|}{n}\r)$, where $0\le f \le 1$ is a bounded continuous function on $[0,1]$. We elaborate on this in Section \ref{sec:limit}. 
The next result gives  asymptotic freeness of independent elliptic matrices  with deterministic matrices.

\begin{theorem}\label{thm2}
Let $A_n^{(1)},A_n^{(2)},\ldots, A_n^{(m)}$ be $m$ independent elliptic random matrices (with possibly different correlations) whose entries 
satisfy Assumption \ref{assu}{\bf (i)}. Then
$$
\bar A_n^{(1)},\ldots,\bar A_n^{(m)}\stackrel{*\mbox{-dist}}{\longrightarrow} e_1,\ldots ,e_{m},
$$
where $e_1,\ldots ,e_{m}$ are free and elliptic. 

In addition, suppose that $A_n^{(1)},A_n^{(2)},\ldots, A_n^{(m)}$ are Gaussian, and $D_n^{(1)},D_n^{(2)}, \ldots, D_n^{(\ell)}$ are  constant matrices such that 
$$
D_n^{(1)},\ldots, D_n^{(\ell)}\stackrel{*\mbox{-dist}}{\longrightarrow} d_1,\ldots, d_{\ell},\;\;\mbox{as $n\to \infty$},
$$
for some $d_1,\ldots,d_{\ell}\in (\mathcal A, \varphi)$. Then 
$$
\bar A_n^{(1)},\ldots,\bar A_n^{(m)},D_n^{(1)},\ldots, D_n^{(\ell)}\stackrel{*\mbox{-dist}}{\longrightarrow} e_1,\ldots ,e_{m},d_1,\ldots, d_{\ell}
$$
where $e_1,\ldots ,e_{m}, \{d_1,\ldots, d_{\ell}\}$ are free.
\end{theorem}

\noindent Now we move to rectangular random matrices.

\begin{theorem}\label{thm:rectangle}
Suppose $X_{p\times n}$ is an elliptic rectangular random matrix whose entries satisfies Assumption \ref{assu}{\bf (ii)}.
If $\frac{p}{n} \to y>0$ as $p\to \infty$, then $\bar X_p=\frac{1}{n}X_{p\times n}X_{n\times p}'$, 
 converges to a free Poisson element of rate $y$. Its expected ESD converges to the 
corresponding Mar\v cenko-Pastur law with parameter $y$.
\end{theorem} 

\noindent Now we claim that independent matrices of the form $\bar X_p$ are asymptotically free. 

\begin{theorem}\label{thm:assymrect}
Let $\bar X_p^{(1)}, \ldots, \bar X_p^{(m)}$ be independent and as described in Theorem \ref{thm:rectangle} with possibly different correlation parameters. If $\frac{p}{n}\to y>0$ as $p\to \infty$, then $\bar X_p^{(1)},\ldots, \bar X_p^{(m)}$ are asymptotically free. 
\end{theorem}

\noindent We now give the Brown measure of product of free elliptic elements.

\begin{theorem}\label{thm:brown}
Let $k\ge 2$ and $e_1,\ldots,e_k$ be $k$ free elliptic elements in $(\mathcal A, \vp)$ with possibly different parameters. Then the Brown measure $\mu_k$ of $e_1\cdots e_k$ is rotationally invariant and can be described by the probabilities  
\begin{align*}
\mu_k(\{z\;:\; |z|\le t\})=\left\{\begin{array}{ll}
t^{\frac{2}{k}} & \mbox{for $t\le 1$}\\
1 & \mbox{for $t>1$}.
\end{array}\r.
\end{align*} 
\end{theorem}

Note that this Brown measure is also the LSD of $E_1\cdots E_k$, where $E_1,\ldots, E_k$ are independent elliptic matrices as defined in \cite{soshnikov}. The result does not hold  for $k=1$, as the Brown measure of elliptic element is uniform distribution on ellipse.
So the condition $k\ge 2$ is crucial.

\section{Proof of Theorem \ref{thm:general}}\label{sec:limit}
We first make the following remark to clarify condition \eqref{eqn:condition}.
For any $\pi\in NC_2(2k)$ and $\epsilon_i \in \{1, *\}$, $1\leq i \leq 2k$, define 
$$T(\pi)=|\{(r,s)\in \pi\;:\; \d_{\epsilon_r\epsilon_s}=1\}|.
$$

\begin{remark}\label{remark}
We show that how in two cases, the limit condition (\ref{eqn:condition}) in the statement of Theorem \ref{thm:general} can be verified.
\begin{enumerate}
 \item[{\bf (I)}] First consider the case where $\rho_i=\rho$ for some $|\rho|\leq 1$ and $\pi\in NC_2(2k)$. Then \eqref{eqn:condition} gives 
\begin{align*}
\lim\frac{1}{n^{k+1}}\sum_{I_{2k}}\prod_{(r,s)\in \pi}\rho_{|i_r-i_{s}|}^{\d_{\e_r\e_{s}}}\d_{i_ri_{s+1}}\d_{i_si_{r+1}}
=&
\lim\frac{1}{n^{k+1}}\sum_{I_{2k}}\prod_{(r,s)\in \pi}\rho^{\d_{\e_r\e_{s}}}\d_{i_ri_{s+1}}\d_{i_si_{r+1}}
\\=&
\lim\frac{1}{n^{k+1}}\sum_{I_{2k}}\prod_{(r,s)\in \pi}\rho^{\d_{\e_r\e_{s}}}\d_{i_ri_{\gamma\pi(r)}}\d_{i_si_{\gamma\pi(s)}}
\\=&\prod_{(r,s)\in \pi}\rho^{\d_{\e_r\e_{s}}}
\lim\frac{1}{n^{k+1}}\sum_{I_{2k}}\prod_{r=1}^{2k}\d_{i_ri_{\gamma\pi(r)}}=\rho^{T(\pi)}.
\end{align*}
The last equality follows from Fact \ref{ft:cardianlity}. Hence the limit condition (\ref{eqn:condition}) holds.

\item[{\bf (II)}] If $\rho_i^{(n)}=f(\frac{i}{n})$, for some continuous function $|f|\le 1$. Then the limit in Condition \eqref{eqn:condition} can be written explicitly in terms of function of $U_1,\ldots, U_k$, depending on $\pi$,
where $U_1,\ldots,U_k$ are i.i.d. uniform random variables in $[0,1]$. We explain it by examples. Note that 
$$
\prod_{(r,s)\in \pi}\d_{i_ri_{s+1}}\d_{i_si_{r+1}}=\prod_{(r,s)\in \pi}\d_{i_ri_{\gamma\pi(r)}}\d_{i_si_{\gamma\pi(s)}}=\prod_{r=1}^{2k}\d_{i_ri_{\gamma\pi(r)}}.
$$ 
The product will be non-zero when all the variables with index from same block of $\gamma\pi $ are equal. Let $\pi=(16)(25)(34)\in NC_2(6)$. Then $\gamma\pi=(1)(26)(35)(4)$, and  $\prod_{(r,s)\in \pi}\d_{i_ri_{s+1}}\d_{i_si_{r+1}}=1$ if $i_2=i_6$ and $i_3=i_5$. Thus, for  $\pi=(16)(25)(34)$, we have
\begin{align*}
&\lim_{n\to \infty}\frac{1}{n^4}\sum_{I_{6}}\prod_{(r,s)\in \pi}\rho_{|i_r-i_{s}|}^{\d_{\e_r\e_{s}}}\d_{i_ri_{s+1}}\d_{i_si_{r+1}}
\\=&\lim_{n\to \infty}\frac{1}{n^4}\sum_{I_{4}}f\l(\l|\frac{i_1}{n}-\frac{i_2}{n}\r|\r)^{\d_{\e_1\e_6}}f\l(\l|\frac{i_2}{n}-\frac{i_3}{n}\r|\r)^{\d_{\e_2\e_5}}f\l(\l|\frac{i_3}{n}-\frac{i_4}{n}\r|\r)^{\d_{\e_3\e_4}}
\\=&\E[f(|U_1-U_2|)^{\d_{\e_1\e_6}}f(|U_2-U_3|)^{\d_{\e_2\e_5}}f(|U_3-U_4|)^{\d_{\e_3\e_4}}],
\end{align*}
as $f$ is a bounded continuous function.  Similarly, it can be shown that, for $\pi=(12)(34)(56)$, $\gamma\pi=(135)(2)(4)(6)$ and $\e_1,\ldots, \e_6\in \{1,*\}$,
\begin{align*}
&\lim_{n\to \infty}\frac{1}{n^4}\sum_{I_{6}}\prod_{(r,s)\in \pi}\rho_{|i_r-i_{s}|}^{\d_{\e_r\e_{s}}}\d_{i_ri_{s+1}}\d_{i_si_{r+1}}
\\=&\E[f(|U_1-U_2|)^{\d_{\e_1\e_2}}f(|U_1-U_3|)^{\d_{\e_3\e_4}}f(|U_1-U_4|)^{\d_{\e_5\e_6}}].
\end{align*}
\end{enumerate}
Similarly the limit in condition \eqref{eqn:condition} can be calculated explicitly in all other cases.
\end{remark}
\noindent Before proving Theorem \ref{thm:general} we state the following lemmas, which will be used in the proof of Theorem \ref{thm:general}. We give the proofs of the lemmas at the end of this section. The first lemma gives the moments of an elliptic element.
\begin{lemma}\label{lem:moments of e}
Let $e$ be an elliptic element with parameter $\rho$ in an NCP $(\mathcal A, \vp)$. Then, for $\epsilon_1, \ldots, \e_p\in \{1,*\}$, 
\begin{align*}
\vp(e^{\epsilon_1}e^{\epsilon_2}\cdots e^{\epsilon_p})=\l\{\begin{array}{ll}
\sum\limits_{\pi\in NC_2(2k)}\rho^{T(\pi)} & \mbox{if $p=2k$}
\vspace{.15cm}\\0 & \mbox{if $p=2k+1$}.
\end{array}\r.
\end{align*}
\end{lemma}

\noindent  The next lemma is  key in  proving Theorem \ref{thm:general}. It will be used repeatedly. We use the notation: 
$$
a'(r,s)=\d_{i_ri_s}\d_{i_{r+1}i_{s+1}} \mbox{ and } b'(r,s)=\d_{i_ri_{s+1}}\d_{i_si_{r+1}}.
$$
\begin{lemma}\label{lem:wigner}
 Consider $a'(r,s)$ and $b'(r,s)$ as defined above. Then
$$
\lim\frac{1}{n^{k+1}}\sum_{I_{2k}}\sum_{\ell=1}^k\sum_{1\le j_1<\cdots<j_{\ell}\le k}a'(r_{j_1},s_{j_1})\cdots a'(r_{j_{\ell}},s_{j_{\ell}})b'(r_{j_{\ell+1}},s_{j_{\ell+1}})\cdots b'(r_{j_k},s_{j_k})=0.
$$
\end{lemma}

  Now we proceed to prove Theorem \ref{thm:general}. We shall use the following notation: 
 $$
X_{ij}^{\e}=\left\{
\begin{array}{ll}
 X_{ij}& \mbox{ if }\  \e=1,\\
 X_{ji} &\mbox{ if } \ \e=*.
\end{array} \r. 
 $$
  
\begin{proof}[Proof of Theorem \ref{thm:general}]
Let $\e_1,\ldots, \e_p\in \{1,*\}$. Then we have 
\begin{align}\label{eqn:thm1.1}
\vp_{n}\l(\bar {A}_n^{\e_1}\cdots \bar A_n^{\e_p}\r)&=\frac{1}{n^{\frac{p}{2}+1}}\E[\Tr(A_n^{\e_1}\cdots A_n^{\e_p})]=\frac{1}{n^{\frac{p}{2}+1}}\sum_{I_p}\E[ a_{i_1i_2}^{\e_1}a_{i_2i_3}^{\e_2}\cdots a_{i_pi_1}^{\e_p}].
\end{align}
Let $w=(i_1,i_2,\ldots, i_p)$ and supp$(w)$ denote the support of $\{i_1,\ldots,i_p\}$, the set consisting of distinct elements from $w$. Let $G_w$ be the graph with the vertex set supp$(w)$ and the non-directed edge set  $\{\{i_k,i_{k+1}\}\;:\; k=1,\ldots, p \}$ where $i_{p+1}=i_1$. Note that by the construction the graph $G_w$ is connected, which further starts and terminates at the same vertex.  Since $G_w$ is connected, for $j<p$, 
\begin{align}\label{eqn:thm1.2}
|\{w\in I_p\; :\; G_w\mbox{ has atmost $j$ distinct edges}\}|=O(n^{j+1}).
\end{align}

By the construction of $G_w$, a maximum of $p$ distinct non-directed edges is possible.
Since $\E [a_{ij}]=0$, each edge has to appear at least twice to have a non-zero  contribution in the right side of \eqref{eqn:thm1.1}. Therefore we have at most $\lfloor \frac{p}{2}\rfloor$  distinct edges, where $\lfloor x \rfloor $ denotes the largest integer not exceeding $x$. In such cases $|\mbox{supp}(w)|$ has at most $\lfloor \frac{p}{2}\rfloor+1$ elements, as $G_w$ is connected.  Therefore by \eqref{eqn:thm1.2} and the fact that the random variables $a_{ij}$ have all moments finite, if $p$ is odd, we have  
\begin{align*}
\lim_{n\to \infty}\frac{1}{n^{\frac{p}{2}+1}}\sum_{I_p}\E[ a_{i_1i_2}^{\e_1}a_{i_2i_3}^{\e_2}\cdots a_{i_pi_1}^{\e_p}]=\lim_{n\to \infty}\frac{O(n^{\lfloor\frac{p}{2}\rfloor+1})}{n^{\frac{p}{2}+1}}=0.
\end{align*}
Saying that $G_w$ has  $k$ distinct edges  is same as saying that there is a pair partition of the $2k$ edges. Therefore by \eqref{eqn:thm1.2} and the fact that the random variables $a_{ij}$ have all moments finite, for $p=2k$, from \eqref{eqn:thm1.1} we get
\begin{align}\label{eqn:gen1}
\vp_{n}(\bar A_n^{\e_1}\cdots \bar A_n^{\e_{2k}})&=\frac{1}{n^{k+1}}\sum_{I_{2k}}\sum_{\pi \in \mathcal P_2(2k)}\prod_{(r,s)\in \pi}\E[a_{i_ri_{r+1}}^{\e_r}a_{i_si_{s+1}}^{\e_s}]+o(1).
\end{align}
Since $A_n$ satisfies Assumption \ref{assu}, we have
\begin{align*}
&\E[a_{i_ri_{r+1}}^{\e_r}a_{i_si_{s+1}}^{\e_s}]\\&=(a'(r,s)+\rho_{|i_r-i_{r+1}|}b'(r,s))\d_{\e_r\e_s} +(\rho_{|i_r-i_{r+1}|}a'(r,s)+b'(r,s))(1-\d_{\e_r\e_s})
\\&=((1-\d_{\e_r\e_s})+\rho_{|i_r-i_{r+1}|}\d_{\e_r\e_s})b'(r,s)+(\d_{\e_r\e_s}+\rho_{|i_r-i_{r+1}|}(1-\d_{\e_r\e_s}))a'(r,s)
\end{align*}
Using this  and  Lemma \ref{lem:wigner}, as $\d_{\e_r\e_s}+\rho_{|i_r-i_{r+1}|}(1-\d_{\e_r\e_s})\le 1$, in \eqref{eqn:gen1} we get
\begin{align*}
\lim_{n\to \infty}\vp_{n}\l(\bar {A}_n^{\e_1}\cdots \bar A_n^{\e_{2k}}\r)
=&\sum_{\pi\in \mathcal P_2(2k)}\lim_{n\to \infty}\frac{1}{n^{k+1}}\sum_{I_{2k}}\prod_{(r,s)\in \pi}((1-\d_{\e_r\e_s})+\rho_{|i_r-i_{r+1}|}\d_{\e_r\e_s})b'(r,s)
\\=&\sum_{\pi\in \mathcal P_2(2k)}\lim_{n\to \infty}\frac{1}{n^{k+1}}\sum_{I_{2k}}\prod_{(r,s)\in \pi}\rho_{|i_r-i_{r+1}|}^{\d_{\e_r\e_s}}\d_{i_ri_{s+1}}\d_{i_si_{r+1}}.
\end{align*} 
Observe that $((1-\d_{\e_r\e_s})+\rho_{|i_r-i_{r+1}|}\d_{\e_r\e_s})=\rho_{|i_r-i_{r+1}|}^{\d_{\e_r\e_s}}$. Since $\rho_i$ are bounded, using Fact \ref{ft:cardianlity}, the right side gives non-zero contribution when $\pi$ is a non-crossing pair matching. Therefore we get 
\begin{align*}
\lim_{n\to \infty}\vp_{n}\l(\bar {A}_n^{\e_1}\cdots \bar A_n^{\e_{2k}}\r)&=\sum_{\pi\in NC_2(2k)}\lim_{n\to \infty}\frac{1}{n^{k+1}}\sum_{I_{2k}}\prod_{(r,s)\in \pi}\rho_{|i_r-i_{s}|}^{\d_{\e_r\e_s}}\d_{i_ri_{s+1}}\d_{i_si_{r+1}},
\end{align*} 
as $i_{r+1}=i_s$. The existence of $*$-distribution limit then follows from Condition \eqref{eqn:condition}.

\noindent {\bf Particular case:} Now suppose $\rho_{i}\equiv \rho$. Therefore, by  Remark \ref{remark} (I), we have
\begin{align*}
\lim_{n\to \infty}\vp_{n}\l(\bar {A}_n^{\e_1}\cdots \bar A_n^{\e_{2k}}\r)&=\sum_{\pi\in NC_2(2k)}\rho^{T(\pi)}=\vp(e^{\e_1}\cdots e^{\e_{2k}}).
\end{align*} 
The last equality follows from Lemma \ref{lem:moments of e}. Hence the result. 
\end{proof}

It remains to prove Lemma \ref{lem:moments of e} and Lemma \ref{lem:wigner}.
\begin{proof}[Proof of Lemma \ref{lem:moments of e}]
First we calculate the free cumulants and mixed free cumulants of $e$ and $e^*$. We have $\kappa_{2}(s_i,s_i)=1$ for $i=1,2$ and $\kappa_{2}(s_1,s_2)=\kappa_2(s_2,s_1)=0$, as $s_1$ and $s_2$ are two standard semi-circular elements and free. Therefore
\begin{align*}
\kappa_2(e,e)&=\kappa_2\l(\sqrt{\frac{1+\rho}{2}}s_1+i\sqrt{\frac{1-\rho}{2}}s_2,\sqrt{\frac{1+\rho}{2}}s_1+i\sqrt{\frac{1-\rho}{2}}s_2\r)
\\&=\frac{1+\rho}{2}\cdot\kappa_2(s_1,s_1)-\frac{1-\rho}{2}\cdot\kappa_2(s_2,s_2)
=\rho.
\end{align*}
Similarly we have $\kappa_2(e^*,e^*)=\rho$, $\kappa_2(e,e^*)=\kappa_2(e^*,e)=1$ and  other free cumulants are zero.
From Fact \ref{ft:moment-cumulant}, we have 
\begin{align*}
\vp(e^{\epsilon_1}e^{\epsilon_2}\cdots e^{\epsilon_{p}})&
=\sum_{\pi\in NC(p)}\prod_{V\in \pi}\kappa_\pi(V)[e^{\epsilon_1},\ldots ,e^{\epsilon_p}].
\end{align*}
Note that only pair partitions will  contribute as the other free cumulants are zero. Therefore if $p$ is odd then  right side of the last equation is zero, as no pair partition is  possible.
And, for $p=2k$, we get
\begin{align*}
\vp(e^{\epsilon_1}e^{\epsilon_2}\cdots e^{\epsilon_{2k}})&= \sum_{\pi\in NC_2(2k)}\prod_{(r,s)\in \pi}\kappa_{\pi}[e^{\epsilon_r} ,e^{\epsilon_s}]=\sum_{\pi\in NC_2(2k)}\rho^{T(\pi)}.
\end{align*}
 Hence the result.
\end{proof}

\begin{proof}[Proof of Lemma \ref{lem:wigner}]
Let $W_n=\l(\frac{Y_{ij}}{\sqrt n}\r)$ be a Wigner matrix, where $(Y_{ij})_{i\le j}$ are i.i.d. $N(0,1)$ and $Y_{ij}=Y_{ji}$. It is well known (e.g., see Theorem 22.16 in \cite{speicherbook}) that $W_n$ converges, as $n\to \infty$, in distribution to a semi-circular element $s$. In other words, $\lim_{n\to \infty} \vp_n(W_n^p)=\vp(s^p)$ for all $p\in \N$. In particular, for $p=2k$, we have
\begin{align}\label{eqn:wigner}
\lim_{n\to\infty}\vp_n(W_n^{2k})=\vp(s^{2k})=\sum_{\pi\in NC_2(2k)}1.
\end{align}
By the trace formula for product of matrices we have
\begin{align*}
\vp_n(W_n^{2k})&=\frac{1}{n^{k+1}}\sum_{I_{2k}}\E[Y_{i_1i_2}Y_{i_2i_3}\cdots Y_{i_{2k}i_1}]
\\ &=\frac{1}{n^{k+1}}\sum_{I_{2k}}\sum_{\pi\in \mathcal P_2(2k)}\prod_{(r,s)\in \pi}\E[Y_{i_ri_{r+1}}Y_{i_si_{s+1}}],
\end{align*}
where the last equality follows from Wick's formula. Again $\E[Y_{i_ri_{r+1}}Y_{i_si_{s+1}}]=a'(r,s)+b'(r,s)$. Recall that $a'(r,s)=\d_{i_ri_{s}}\d_{i_{r+1}i_{s+1}}$ and $b'(r,s)=\d_{i_ri_{s+1}}\d_{i_si_{r+1}}$.
Therefore we have
\begin{align*}
\vp_n(W_n^{2k})=\frac{1}{n^{k+1}}\sum_{\pi\in \mathcal P_2(2k)}\sum_{I_{2k}}\prod_{j=1}^k(a'(r_{j},s_{j})+b'(r_{j},s_{j})),
\end{align*}
where  $\pi=\{(r_j,s_j)\; :\; j=1,\ldots, k\}$. Let $a_j'=a'(r_{j},s_{j})$ and $b_j'=b(r_{j},s_{j}')$, then by expanding the product we get
\begin{align}\label{eqn:wigner2}
\vp_n(W_n^{2k})=\frac{1}{n^{k+1}}\sum_{I_{2k}}\sum_{\pi\in \mathcal P_2(2k)}\l( \prod_{j=1}^kb_j'+\sum_{\ell=1}^k\sum_{1\le j_1<\ldots<j_{\ell}\le k}a_{j_1}'\cdots a_{j_{\ell}}'b_{j_{\ell+1}}'\cdots b_{j_k}'\r),
\end{align}
 Again we have 
\begin{align}\label{eqn:contribuingwig}
\lim_{n\to \infty}\frac{1}{n^{k+1}}\sum_{\pi\in \mathcal P_2(2k)}\sum_{I_{2k}}\prod_{j=1}^kb_j'&=\lim_{n\to \infty}\frac{1}{n^{k+1}}\sum_{\pi\in \mathcal P_2(2k)}\sum_{I_{2k}}\prod_{j=1}^k\d_{i_{r_j}i_{s_j+1}}\d_{i_{s_j}i_{r_j+1}}\nonumber
\\&=\lim_{n\to \infty}\frac{1}{n^{k+1}}\sum_{\pi\in \mathcal P_2(2k)}\sum_{I_{2k}}\prod_{r=1}^{2k}\d_{i_{r}i_{\gamma \pi(r)}}\nonumber
\\&=\lim_{n\to \infty}\frac{1}{n^{k+1}}\sum_{\pi\in \mathcal P_2(2k)} n^{|\gamma \pi|}
=\sum_{\pi\in NC_2(2k)}1,
\end{align}
where the last equality is a consequence of Fact \ref{ft:cardianlity}. The result now follows from \eqref{eqn:wigner}, \eqref{eqn:wigner2} and \eqref{eqn:contribuingwig}.
\end{proof}

\section{Proof of Theorem \ref{thm2}}\label{sec:asypt}

We first state a fact and a lemma. 
The proof of the lemma is given at the end of this section.
\begin{fact}\label{ft:freenessproperty}
Let $(\mathcal A,\vp)$ be an NCP and  $a_1,\ldots, a_n,b_1,\ldots, b_n\in \mathcal A$. Then  $\{a_1,\ldots,a_n\}$ and $\{b_1,\ldots,b_n\}$ are free if and only if, for all $k\in \N$, $\e_1,\ldots,\e_k,\tau_1,\ldots, \tau_k\in \{0,1,*\}$ and $i_1,\ldots, i_k,j_1,\ldots, j_k\in \{1,\ldots, n\}$,
$$
\vp(a_{i_1}^{\e_1}b_{j_1}^{\tau_1}a_{i_2}^{\e_2}b_{j_2}^{\tau_2}\cdots a_{i_k}^{\e_k}b_{j_k}^{\tau_k})=\sum_{\pi\in NC(k)}\kappa_{\pi}[a_{i_1}^{\e_1},\ldots,a_{i_k}^{\e_k}]\cdot \vp_{K(\pi)}[b_{j_1}^{\tau_1},\ldots,b_{j_k}^{\tau_k}],
$$
where $K(\pi)$ denotes the Kreweras complement of $\pi$.
\end{fact}

\noindent For the proof of the forward direction of Fact \ref{ft:freenessproperty} we refer to \cite{speicherbook}, p. 226. The other direction follows easily. For the next lemma, we need the following notation:
\begin{align*}
\mbox{$u_{\ell}'=\d_{i_{r_{\ell}}i_{s_{\ell}}}\d_{j_{r_{\ell}}j_{s_{\ell}}}$, $v_{\ell}'=\d_{i_{r_{\ell}}j_{s_{\ell}}}\d_{j_{r_{\ell}}i_{s_{\ell}}}$ and }
\end{align*} 
$$
S_t=\{(\{\ell_1,\ldots, \ell_t\},\{\ell_{t+1},\ldots, \ell_k\})\;:\;\{\ell_1,\ldots, \ell_t\}\cup \{\ell_{t+1},\ldots, \ell_k\}=\{1,\ldots, k\}\}.
$$

\begin{lemma}\label{lem:forfreeness}
Let $D_n$ be a sequence of deterministic matrices which converges in  $*$-distribution  to an element  $d$. Then, for $\pi=\{(r_1,s_1),\ldots, (r_k,s_k)\}$,
\begin{align}\label{eqn:forfree}
\lim_{n\to \infty}\frac{1}{n^{k+1}}\sum_{I_{2k},J_{2k}}\sum_{\pi\in \mathcal P_2(2k)}\l(\sum_{t=1}^{k}\sum_{S_t}u_{\ell_1}'\cdots u_{\ell_t}'v_{\ell_{t+1}}'\cdots v_{\ell_{k}}'
\r)\prod_{\ell=1}^{k}D^{(\ell)}_{j_{\ell}i_{\ell+1}}=0,
\end{align}
where  $D_n^{(\ell)}=D_n^{\tau_1}\cdots D_n^{\tau_{q_{\ell}}}$ for $q_i\in \N$ and $\tau_i\in \{0,1,*\}$.
\end{lemma}


\begin{proof}[Proof of Theorem \ref{thm2}]
 Let $\tau_1,\ldots ,\tau_p\in \{1,2,\ldots, m\}$ and $\e_1,\ldots,\e_p\in \{1,*\}$. Then 
$$
\vp_n(\bar A_n^{(\tau_1)\epsilon_1}\bar A_n^{(\tau_2)\epsilon_2}\cdots \bar A_n^{(\tau_{p})\epsilon_{p}})=\frac{1}{n^{\frac{p}{2}+1}}\sum_{I_{2k}}\E[a_{i_1i_2}^{(\tau_1)\e_1}a_{i_2i_3}^{(\tau_2)\e_2}\cdots a_{i_{p}i_1}^{(\tau_{p})\e_{p}}].
$$
Since the random variables $a_{ij}$ have mean zero and all moments finite, using the arguments as in the proof of Theorem \ref{thm:general} we have
$$
\lim_{n\to \infty}\vp_n(\bar A_n^{(\tau_1)\epsilon_1}\bar A_n^{(\tau_2)\epsilon_2}\cdots \bar A_n^{(\tau_p)\epsilon_p})=0,
$$
when $p$ is odd. Similarly, for $p=2k$, we have 
\begin{align*}
\frac{1}{n^{k+1}}\sum_{I_{2k}}\E[a_{i_1i_2}^{(\tau_1)\e_1}a_{i_2i_3}^{(\tau_2)\e_2}\cdots a_{i_{2k}i_1}^{(\tau_{2k})\e_{2k}}]=\frac{1}{n^{k+1}}\sum_{I_{2k}}\sum_{\pi \in \mathcal P_2(2k)}\prod_{(r,s)\in \pi}\E[a_{i_ri_{r+1}}^{(\tau_r)\e_r}a_{i_si_{s+1}}^{(\tau_s)\e_s}]+o(1).
\end{align*}
Since $A_n^{(1)}, \ldots, A_n^{(m)}$ are independent elliptic matrices, there exist $ \rho_1,\ldots, \rho_m$ such that $\E[a_{ij}^{(\tau)}a_{ji}^{(\tau)}]=\rho_{\tau}$ for $\tau=1,\ldots, m$ and $1\le i,j \le n$.
Let $a_{\tau}(r,s)=(\rho_{\tau}+(1-\rho_{\tau})\d_{\e_r\e_s})a'(r,s)$ and $b_{\tau}(r,s)=(1-(1-\rho_{\tau})\d_{\e_r\e_s})b'(r,s)$. Then
\begin{align*}
\E[a_{i_ri_{r+1}}^{(\tau_r)\e_r}a_{i_si_{s+1}}^{(\tau_s)\e_s}]=(a_{\tau}(r,s)+b_{\tau}(r,s))\d_{\tau_r\tau_s}^{\tau_{rs}}, \mbox{ where  $\tau_{rs}\in \{1,\ldots, m\}$,}
\end{align*}
where $\d_{\tau_r\tau_s}^{\tau_{rs}}=1$ when $\tau_r=\tau_s=\tau_{rs}$ and zero otherwise.
Note that $a(r,s)\le a'(r,s)$ and $b(r,s)\le b'(r,s)$. Therefore combining all and using Lemma \ref{lem:wigner} we have 
\begin{align}\label{eqn:twomatrixfree1}
\lim_{n\to \infty}\vp_n(A_n^{(\tau_1)\epsilon_1}A_n^{(\tau_2)\epsilon_2}\cdots A_n^{(\tau_{2k})\epsilon_{2k}})&=\lim_{n\to \infty}\frac{1}{n^{k+1}}\sum_{I_{2k}}\sum_{\pi \in \mathcal P_2(2k)}\prod_{(r,s)\in \pi}b_{\tau_{rs}}(r,s)\d_{\tau_r\tau_s}^{\tau_{rs}}\nonumber
\\&=\sum_{\pi \in NC_2(2k)}\rho_1^{T_1(\pi)}\ldots \rho_m^{T_m(\pi)}\prod_{(r,s)\in \pi}\d_{\tau_r\tau_{s}},
\end{align}
where $T_{\tau}(\pi):=|\{(r,s)\in \pi \;:\; \d_{\e_r\e_s}=1, \tau_r=\tau_s=\tau\}|$.
 
 On the other hand,  if $e_1,\ldots, e_m$ are free elliptic  with parameters $\rho_1,\ldots, \rho_m$ respectively, then by Fact \ref{ft:moment-cumulant} we have 
\begin{align*}
\vp(e_{\tau_1}^{\e_1}e_{\tau_2}^{\e_2}\cdots e_{\tau_{2k}}^{\e_{2k}})=\sum_{\pi\in NC(2k)}\prod_{V\in \pi}\kappa(V)[e_{\tau_1}^{\e_1},\ldots,e_{\tau_{2k}}^{\e_{2k}}].
\end{align*}
Moreover, $\kappa(e_i,e_i)=\kappa(e_i^*,e_i^*)=\rho_i$ and $\kappa(e_i,e_i^*)=\kappa(e_i^*,e_i)=1$ for $i=1,2$ and all other free cumulants are zero. Therefore we get
\begin{align}\label{eqn:twomatrixfree2}
\vp(e_{\tau_1}^{\e_1}e_{\tau_2}^{\e_2}\cdots e_{\tau_{2k}}^{\e_{2k}})=\sum_{\pi \in NC_2(2k)}\rho_1^{T_1(\pi)}\ldots \rho_m^{T_m(\pi)}\prod_{(r,s)\in \pi}\d_{\tau_r\tau_{s}}.
\end{align}
From \eqref{eqn:twomatrixfree1}, \eqref{eqn:twomatrixfree2} and Fact \ref{ft:freenessproperty} we conclude that $A_n^{(1)},\ldots, A_n^{(m)}$ converge, as $n\to \infty$, to $e_1,\ldots, e_m$ in $*$-distribution sense and are  asymptotically free.

\vspace{.2cm}
\noindent {\it Proof of the second part:} Suppose now that the entries of $A_n$ are Gaussian. For the ease of writing we give the proof only for $m=1$ and $\ell=1$. The general case can be tackled in a similar way.
Let $\e_1,\ldots, \e_p \in \{1,*\}$, and $D_n^{(\ell)}=D_n^{\e_1}\cdots D_n^{\e_{q_{\ell}}}$ for $q_1,\ldots, q_p\in \N\cup \{0\}$ with the convention that $D^{(\ell)}=I$ if $q_{\ell}=0$. Then the mixed $*$-moment of $\bar A_n$ and $D_n$ are  given by
\begin{align*}
\vp_n(\bar A_n^{\epsilon_1}D_n^{(1)}\cdots \bar A_n^{\e_p}D_n^{(p)})=\frac{1}{n^{\frac{p}{2}}+1}\sum_{I_p,J_p}\E[a_{i_1j_1}^{\e_1}a_{i_2j_2}^{\e_2}\cdots a_{i_pj_{p}}^{\e_p}]\prod_{\ell=1}^pD^{(\ell)}_{j_{\ell}i_{\ell+1}}.
\end{align*}
The entries $a_{ij}$ are mean zero Gaussian random variables. Therefore, by Wick's formula, if $p$ is odd,
 \begin{align*}
 \vp_n(\bar A_n^{\epsilon_1}D_n^{(1)}\cdots \bar A_n^{\e_p}D_n^{(p)})=0.
 \end{align*}
Let $p=2k$. Then by Wick's formula   we have
\begin{align*}
\vp_n(\bar A_n^{\epsilon_1}D_n^{(1)}\cdots \bar A_n^{\e_{2k}}D_n^{(2k)})
&=\frac{1}{n^{k+1}}\sum_{I_{2k},J_{2k}}\sum_{\pi\in \mathcal P_2(2k)}\prod_{(r,s)\in \pi}\E[a_{i_rj_r}^{\e_r}a_{i_sj_s}^{\e_s}]\prod_{\ell=1}^{2k}D^{(\ell)}_{j_{\ell}i_{\ell+1}}
\\&=\frac{1}{n^{k+1}}\sum_{I_{2k},J_{2k}}\sum_{\pi\in \mathcal P_2(2k)}\prod_{(r,s)\in \pi}(u(r,s)+v(r,s))\prod_{\ell=1}^{2k}D^{(\ell)}_{j_{\ell}i_{\ell+1}},
\end{align*}
where $u(r,s)=(\rho+(1-\rho)\d_{\e_r\e_s})\d_{i_ri_s}\d_{j_rj_s}$ and $v(r,s)=(1-(1-\rho)\d_{\e_r\e_s})\d_{i_rj_s}\d_{j_ri_s}$. Note that $u(r,s)\le \d_{i_ri_s}\d_{j_rj_s}$ and $v(r,s) \le \d_{i_rj_s}\d_{j_ri_s}$. Then applying Lemma \ref{lem:forfreeness} in the last equation we get
\begin{align*}
&\lim_{n\to\infty}\vp_n(\bar A_n^{\epsilon_1}D_n^{(1)}\cdots \bar A_n^{\e_{2k}}D_n^{(2k)})
\\=&\frac{1}{n^{k+1}}\sum_{I_{2k},J_{2k}}\sum_{\pi\in \mathcal P_2(2k)}\prod_{(r,s)\in \pi}v(r,s)\prod_{\ell=1}^{2k}D^{(\ell)}_{j_{\ell}i_{\ell+1}}
\\=&\lim_{n\to\infty}\frac{1}{n^{k+1}}\sum_{\pi\in\mathcal P_2(2k)}\prod_{(r,s)\in \pi}(1-(1-\rho)\d_{\e_r\e_s})\sum_{I_{2k},J_{2k}}\prod_{\ell=1}^{2k}\d_{i_{{\ell}}j_{\pi({\ell})}}
\prod_{\ell=1}^{2k}D^{(\ell)}_{j_{\ell}i_{\gamma(\ell)}},
\\=&\sum_{\pi\in\mathcal P_2(2k)}\rho^{T(\pi)}\cdot \lim_{n\to\infty}\frac{1}{n^{k+1}}\sum_{J_{2k}}\prod_{\ell=1}^{2k}D^{(\ell)}_{j_{\ell}j_{\pi\gamma(\ell)}}.
\end{align*}
 Again we have 
\begin{align*}
\sum_{J_{2k}}\prod_{\ell=1}^{2k}D^{(\ell)}_{j_{\ell}j_{\pi\gamma(\ell)}}=\tr_{\pi \gamma}[D_n^{(1)},\ldots,D_n^{(2k)}]\cdot n^{|\gamma\pi|}.
\end{align*}
Since $D_n\stackrel{*\mbox{-dist}}{\longrightarrow} d$, it follows that 
$$
\tr_{\pi \gamma}[D_n^{(1)},\ldots,D_n^{(2k)}]\to \vp_{\pi \gamma}[d^{(1)},\ldots,d^{(2k)}],\;\;\mbox{as $n\to \infty$,}
$$ 
 where $d^{(s)}=d^{\e_1}\cdots d^{\e_{q_s}}$ and $d^{(s)}=1$ if $q_s=0$ for $s=1,\ldots, 2k$ . We explain the notation $\vp_{\pi\gamma}$ by an example. If $\pi \gamma=(135)(2)(4)(6)$, then $\vp_{\pi \gamma}[d^{(1)}, \ldots, d^{(6)}]=\vp(d^{(1)}d^{(3)}d^{(5)})\vp(d^{(2)})\vp(d^{(4)})\vp(d^{(6)})$. Therefore by Fact \ref{ft:cardianlity}, $|\gamma \pi|=k+1$ when $\pi$ is a non-crossing pair partition, we get 
\begin{align}\label{eqn:traceformula}
\lim_{n\to\infty}\vp_n(\bar A_n^{\epsilon_1}D_n^{(1)}\cdots \bar A_n^{\e_{2k}}D_n^{(2k)})
=\sum_{\pi\in NC_2(k)}\rho^{T(\pi)}\vp_{\pi \gamma}[d^{(1)},\ldots,d^{(2k)}].
\end{align}
Moreover, if $e$ and $d$ are free, then from Fact \ref{ft:freenessproperty} we have
\begin{align}\label{eqn:trace2}
\vp(e^{\e_1}d^{(1)}\cdots e^{\e_{2k}}d^{(2k)})&=\sum_{\pi\in NC(2k)}\varphi(e^{\e_1}\cdots e^{\e_k})\vp_{K(\pi)}[d^{(1)},\ldots,d^{(2k)}]\nonumber
\\&=\sum_{\pi\in NC_2(2k)}\rho^{T(\pi)}\vp_{\pi \gamma}[d^{(1)},\ldots,d^{(2k)}].
\end{align}
The result now follows from \eqref{eqn:traceformula}, \eqref{eqn:trace2} and Fact \ref{ft:freenessproperty}.
\end{proof}

It remains to prove Lemma \ref{lem:forfreeness}. 
We first exemplify the proof for a few special cases. Let $p=4$ and $\pi=(12)(34)$. Then, for self-matching among $i$'s and $j$'s,
\begin{align*}
\lim_{n\to \infty}\frac{1}{n^3}
\sum_{I_4,J_4}\d_{i_1i_2}\d_{j_1j_2}\d_{i_3i_4}\d_{j_3j_4}\prod_{1}^4D^{(\ell)}_{j_{\ell}i_{\ell+1}}
=&\lim_{n\to \infty}\frac{1}{n^3}\sum_{i_1,j_1,i_3,j_3}D^{(1)}_{j_1i_1}D^{(2)}_{j_1i_3}D^{(3)}_{j_3i_3}D^{(1)}_{j_3i_1}
\\=\lim_{n\to \infty}\frac{1}{n^3}\l\{\tr(D_n^{(1)*}D_n^{(2)}D_n^{(3)*}D_n^{(4)})\cdot n\r\}=&0,
\end{align*}
as $D_n^{(i)}$ converge in $*$-distribution as $n\to \infty$.
Similarly, for one self-matching and one non-self-matching among $i$'s and $j$'s, we have 
\begin{align*}
\lim_{n\to \infty}\frac{1}{n^3}
\sum_{I_4,J_4}\d_{i_1i_2}\d_{j_1j_2}\d_{i_3j_4}\d_{j_3i_4}\prod_{1}^4D^{(\ell)}_{j_{\ell}i_{\ell+1}}
=&\lim_{n\to \infty}\frac{1}{n^3}\sum_{i_1,j_1,j_3,j_4}D^{(1)}_{j_1i_1}D^{(2)}_{j_1j_4}D^{(3)}_{j_3j_3}D^{(1)}_{j_4i_1}
\\=\lim_{n\to \infty}\frac{1}{n^3}\l\{\tr(D_n^{(1)*}D_n^{(2)}D_n^{(4)})\tr(D_n^{(3)})\cdot n^2\r\}=&0.
\end{align*}
In the renaming case, there is no self-matching among $i$'s and $j$'s. Then 
\begin{align*}
&\lim_{n\to \infty}\frac{1}{n^3}
\sum_{I_4,J_4}\d_{i_1j_2}\d_{j_1i_2}\d_{i_3j_4}\d_{j_3i_4}\prod_{1}^4D^{(\ell)}_{j_{\ell}i_{\ell+1}}
=\lim_{n\to \infty}\frac{1}{n^3}\sum_{j_1,j_2,j_3,j_4}D^{(1)}_{j_1j_3}D^{(2)}_{j_2j_2}D^{(3)}_{j_3j_1}D^{(1)}_{j_4j_4}
\\=&\lim_{n\to \infty}\frac{1}{n^3}\l\{\tr(D_n^{(1)}D_n^{(3)})\tr(D_n^{(2)})\tr(D_n^{(4)})\cdot n^3\r\}=\vp(d_1d_3)\vp(d_2)\vp(d_4).
\end{align*}
The calculations show that the limit is non-zero when there is no self-matching among $i$'s and $j$'s.  In general, Lemma \ref{lem:forfreeness} says that the limit is non-zero when there is no self-matching among $i$'s and $j$'s.

\begin{proof}[Proof of Lemma \ref{lem:forfreeness}]
We use induction. First suppose $k=1$. Then  
\begin{align*}
\lim_{n\to \infty}\frac{1}{n^2} \sum_{I_{2},J_2}\d_{i_1i_2}\d_{j_1j_2}D^{(1)}_{j_1i_2}D^{(2)}_{j_2i_1}
&=\lim_{n\to \infty}\frac{1}{n^2}\sum_{i_1,j_1}D^{(1)}_{j_1i_1}D^{(2)}_{j_1i_1}
\\&=\lim_{n\to \infty}\frac{1}{n^2}\cdot \tr(D^{(1)}D^{(2)*}).n=0,
\end{align*}
since  $\tr(D^{(1)}D^{(2)*})\stackrel{n\to \infty}{\to}\vp(d_1d_2^*)$ from the fact that $\{D^{(1)},D^{(2)}\}\stackrel{*\mbox{-dist}}{\to}\{d_1,d_2\}$.

Next we show \eqref{eqn:forfree} holds for $k+1$ if it holds  for $k$.
Let $\pi\in \mathcal P_2(2(k+1))$ and $\pi=(r_1,s_1)\cdots (r_{k+1},s_{k+1})$, with the convention that $r_i \leq  s_i$ for all $i$.
Note that there is at least one self-matching pair in $\{i_1,\ldots, i_{2k+2}\}$ and $\{j_1,\ldots, j_{2k}\}$. Without loss of generality we assume that the self-matching in $(r_1,s_1)$ pairs, i.e.,  $u_{\ell_1}'=\d_{i_{r_1}i_{s_1}}\d_{j_{r_1}j_{s_1}}=1$. Then
\begin{align}\label{eqn:free1}
&\lim_{n\to \infty}\frac{1}{n^{k+2}}\sum_{I_{2k+2},J_{2k+2}}\sum_{t=1}^{k+1}\sum_{S_t}u_{\ell_1}'\cdots u_{\ell_t}'v_{\ell_{t+1}}'\cdots v_{\ell_{k+1}}'
\prod_{\ell=1}^{2k+2}D^{(\ell)}_{j_{\ell}i_{\ell+1}}
\\=&\lim_{n\to \infty}\frac{1}{n^{k+2}}\sum_{I_{2k+2},J_{2k+2}}\sum_{t=1}^{k+1}\sum_{S_t}\d_{i_{r_1}i_{s_1}}
\d_{j_{r_1}j_{s_1}}u_{\ell_2}'\cdots u_{\ell_t}'v_{\ell_{t+1}}'\cdots v_{\ell_{k+1}}'
\prod_{\ell=1}^{2k+2}D^{(\ell)}_{j_{\ell}i_{\ell+1}}.\nonumber
\end{align}
{\bf Case I:} Let  $r_1=s_1-1$ (similar for $r_1=1$ and $s_1=2k$). Then we have 
\begin{align*}
\sum_{i_{r_1},i_{s_1}, j_{r_1}, j_{s_1}}\d_{i_{r_1}i_{s_1}}\d_{j_{r_1}j_{s_1}}D^{(r_1-1)}_{j_{r_1-1}i_{r_1}}
D^{(r_1)}_{j_{r_1}i_{r_1+1}}D^{(s_1)}_{j_{s_1}i_{s_1+1}}
&=\sum_{i_{r_1},j_{r_1}}D^{(r_1-1)}_{j_{r_1-1}i_{r_1}}
D^{(r_1)*}_{i_{r_1}j_{r_1}}D^{(s_1)}_{j_{r_1}i_{s_1+1}}
\\&=(D^{(r_1-1)}D^{(r_1)*}
D^{(s_1)})_{j_{r_1-1}i_{s_1+1}}.
\end{align*}
Since $D^{(r_1-1)}D^{(r_1)*}D^{(s_1)}$ is of the form $D^{\tau_1}\cdots D^{\tau_q}$ for some positive integer $q$ and $\tau_1,\ldots, \tau_q\in \{0,1,*\}$. Renaming rest of the $2k$ variables of $i$'s,  $j$'s and $D^{i}$'s,  \eqref{eqn:free1} becomes of the form 
\begin{align}\label{eqn:free2}
\lim_{n\to \infty}\frac{1}{n^{k+2}}\sum_{I_{2k},J_{2k}}\sum_{t=0}^{k}\sum_{S_t}u_{\ell_1}'\cdots u_{\ell_t}'v_{\ell_{t+1}}'\cdots v_{\ell_{k}}'
\prod_{\ell=1}^{2k}D^{(\ell)}_{j_{\ell}i_{\ell+1}}.
\end{align}
Let $t\ge 1$, i.e., there exists at least one self-matching within $i$'s and $j$'s. Therefore the limit \eqref{eqn:free2} is zero by the induction hypothesis. Let  $t=0$, i.e., there  is no self-matching in $i$'s or $j$'s. Then also the limit \eqref{eqn:free2} is zero because the maximum number of blocks among $D^{(i)}$'s is $k+1$, using Fact \ref{ft:cardianlity}.

\noindent {\bf Case II:} Let $r_1\neq s_1-1$ (except for $r_1=1$ and $s_1=2k$). Then we have 
\begin{align*}
&\sum_{i_{r_1},i_{s_1}, j_{r_1}, j_{s_1}}\d_{i_{r_1}i_{s_1}}\d_{j_{r_1}j_{s_1}}D^{(r_1-1)}_{j_{r_1-1}i_{r_1}}
D^{(s_1-1)}_{j_{s_1-1}i_{s_1}}
D^{(r_1)}_{j_{r_1}i_{r_1+1}}D^{(s_1)}_{j_{s_1}i_{s_1+1}}
\\&=(D^{(r_1-1)}D^{(s_1-1)*})_{j_{r_1-1}j_{s_1-1}}(D^{(r_1)*}D^{(s_1)})_{i_{r_1+11}i_{s_1+11}}.
\end{align*}
Therefore four matrices reduced to two matrices, each of the form $D^{\tau_1}\cdots D^{\tau_q}$ for some positive integer $q$ and $\tau_1,\ldots, \tau_q\in \{0,1,*\}$. Again we have reduction of four variables $i_{r_1},i_{s_1},j_{r_1},j_{s_1}$. By renaming the variables properly, the expression \eqref{eqn:free1} reduces to the form of \eqref{eqn:free2}, and which is zero by the previous arguments, i.e.,
\begin{align*}
\lim_{n\to \infty}\frac{1}{n^{k+2}}\sum_{I_{2k+2},J_{2k+2}}\sum_{t=1}^{k+1}\sum_{S_t}u_{\ell_1}'\cdots u_{\ell_t}'v_{\ell_{t+1}}'\cdots v_{\ell_{k+1}}'
\prod_{\ell=1}^{2k+2}D^{(\ell)}_{j_{\ell}i_{\ell+1}}=0.
\end{align*}
This completes the proof by the principle of induction.
\end{proof}

\section{Proofs of Theorem \ref{thm:rectangle} and Theorem \ref{thm:assymrect}}\label{sec:rect}
 The following result will be used in the proof of Theorem \ref{thm:rectangle}. This result holds in more general settings. However the restricted version is enough for our purpose. For its proof we refer to Theorem 5.5 in \cite{bosehazra}.

\begin{result}\label{res:rectg}
Suppose $Y_{p\times n} $ is a $p\times n$ rectangular matrix whose entries are i.i.d $N(0,1)$ and  
$\frac{p}{n}\to y>0$ as $p\to \infty$. Then the ESD $\mu_p$ of 
$\bar Y_p=\frac{1}{n}Y_{p\times n}Y_{n\times p}^*$ 
converges to the Mar\v cenko-Pastur law, $\mu$ given by 
\begin{align*}
\mu(A)=\l\{ \begin{array}{lr}
(1-\frac{1}{y})_{0\in A}+\nu(A) & \mbox{if $y>1$},\\
\nu(A) & \mbox{if $y\le 1$},
\end{array} \r.
\end{align*} 
where 
$$
d\nu(x)=\l\{\begin{array}{ll}
\frac{1}{2\pi xy}\sqrt{(b-x)(x-a)} & \mbox{if $a\le x \le b$},
\\ 0 & \mbox{otherwise}
\end{array} \r.
$$
where $a=a(y)=(1-\sqrt y)^2$ and $b=b(y)=(1+\sqrt y)^2$.
\end{result}

\begin{proof}[Proof of Theorem \ref{thm:rectangle}]
By Result \ref{res:rectg} it is enough to show that, for all $k$, 
 \begin{align}\label{eqn:rect00}
 \lim_{p\to\infty} \frac{1}{p}\E[\Tr (\bar X_p^k)]=\lim_{p\to \infty} \frac{1}{p}\E[\Tr(\bar Y_p^k)]. 
 \end{align}
 The the left side of \eqref{eqn:rect00} equals
\begin{align}\label{eqn:rect0}
\frac{1}{p}\E[\Tr(\bar X_p^k)]=\frac{1}{p.n^k}\sum_{I_{2k}'}\E[X_{i_1i_2}^{\e_1}X_{i_2i_3}^{\e_2}\ldots X_{i_{2k}i_1}^{\e_{2k}}]
\end{align}
where $I_{2k}'=\{(i_1,\ldots, i_{2k})\; :\; 1\le i_{2m} \le n, 1\le i_{2m-1}\le p, m=1,\ldots, k\}$, $\e_{2m}=*$ and $\e_{2m-1}=1$ for $m=1,\ldots,k$.  Using the same arguments as in the proof of Theorem \ref{thm:general} we have 
\begin{align*}
\frac{1}{p.n^k}\sum_{I_{2k}'}\E[X_{i_1i_2}^{\e_1}X_{i_2i_3}^{\e_2}\ldots X_{i_{2k}i_1}^{\e_{2k}}]=\frac{1}{p.n^k}\sum_{I_{2k}'}\sum_{\pi\in \mathcal P_2(2k)}\prod_{(r,s)\in \pi}\E[X_{i_ri_{r+1}}^{\e_r}X_{i_si_{s+1}}^{\e_s}]+o(1).
\end{align*}
Using the properties of the random variables, we have 
\begin{align*}
\E[X_{i_ri_{r+1}}^{\e_r}X_{i_si_{s+1}}^{\e_s}]=&(\d_{i_ri_s}\d_{i_{r+1}i_{s+1}}+\rho \d_{i_ri_{s+1}}\d_{i_{r+1}i_{s}}\d_{\{i_r,i_{r+1}\le \min{\{p,n\}}\}})\d_{\e_r\e_s}
\\& + (\d_{i_ri_{s+1}}\d_{i_{r+1}i_{s}}+\rho \d_{i_ri_s}\d_{i_{r+1}i_{s+1}}\d_{\{i_r,i_{r+1}\le \min{\{p,n\}}\}})(1-\d_{\e_r\e_s})
\\ =&(\d_{\e_r\e_s}+\rho(1-\d_{\e_r\e_s})\d_{\{i_r,i_{r+1}\le \min{\{p,n\}}\}})\d_{i_ri_s}\d_{i_{r+1}i_{s+1}}
\\& +(\rho \d_{\e_r\e_s}\d_{\{i_r,i_{r+1}\le \min{\{p,n\}}\}}+(1-\d_{\e_r\e_s}))\d_{i_ri_{s+1}}\d_{i_{r+1}i_{s}}
\\=&g(r,s)+f(r,s), \mbox{ say.}
\end{align*}
 Note that $g(r,s)\le \d_{i_ri_s}\d_{i_{r+1}i_{s+1}}$ and $f(r,s)\le \d_{i_ri_{s+1}}\d_{i_{r+1}i_{s}}$. Therefore, as $p/n\to  y>0$, by Lemma \ref{lem:wigner}, we have 
\begin{align}\label{eqn:rect1}
\lim_{p\to \infty}\frac{1}{p.n^k}\sum_{I_{2k}'}\E[X_{i_1i_2}^{\e_1}X_{i_2i_3}^{\e_2}\ldots X_{i_{2k}i_1}^{\e_{2k}}]&=\lim_{p\to \infty}\frac{1}{p.n^k}\sum_{I_{2k}'}\sum_{\pi\in \mathcal P_2(2k)}\prod_{(r,s)\in \pi}f(r,s)\nonumber
\\&=\sum_{\pi \in NC_2(2k)}\lim_{n\to \infty}\frac{1}{p.n^k}\sum_{I_{2k}'}\prod_{(r,s)\in \pi}f(r,s).
\end{align}
If $\pi\in NC_2(2k)$ then $\d_{\e_r\e_s}=0$ for $(r,s)\in \pi$, as one of $r$ and $s$ appears at an  odd place and the other one appears at an  even place. Therefore 
$$
f(r,s)=\d_{i_ri_{s+1}}\d_{i_{r+1}i_{s}},\;\; \mbox{ for $(r,s)\in \pi$}.
$$
Therefore, for $\pi\in NC_2(2k)$, we get 
\begin{align*}
\prod_{(r,s)\in\pi}f(r,s)=\prod_{r=1}^{2k}\d_{i_ri_{\gamma\pi(r)}}.
\end{align*}
Therefore from \eqref{eqn:rect0} and \eqref{eqn:rect1} we get 
\begin{align}\label{eqn:22}
\lim_{p\to \infty}\frac{1}{p}\E[\Tr(\bar X_p^k)]=\sum_{\pi \in NC_2(2k)}\lim_{n\to \infty}\frac{1}{p.n^k}\sum_{I_{2k}'}\prod_{r=1}^{2k}\d_{i_ri_{\gamma\pi(r)}}.
\end{align}

On the other hand, similar to \eqref{eqn:rect0}, the right side of \eqref{eqn:rect00} equals
\begin{align*}
\frac{1}{pn^k}\sum_{I_{2k}'}\E[Y_{i_1i_2}^{\e_1}\cdots Y_{i_{2k}i_1}^{\e_{2k}}]=\frac{1}{pn^k}\sum_{I_{2k}'}
\sum_{\pi\in \mathcal P_2(2k)}\prod_{(r,s)\in\pi}\E[Y_{i_ri_{r+1}}^{\e_r}Y_{i_si_{s+1}}^{\e_s}].
\end{align*}
The last equality follows from Wick's formula. Since $Y_{ij}$ are i.i.d. $N(0,1)$,
\begin{align*}
\E[Y_{i_ri_{r+1}}^{\e_r}Y_{i_si_{s+1}}^{\e_s}]=\d_{i_ri_s}\d_{i_{r+1}i_{s+1}}
\d_{\e_r\e_s}+\d_{i_ri_{s+1}}\d_{i_si_{r+1}}(1-\d_{\e_r\e_s}).
\end{align*}
Therefore by Lemma \ref{lem:wigner} we get 
\begin{align*}
&\lim_{p\to \infty}\frac{1}{pn^k}\sum_{I_{2k}'}
\sum_{\pi\in \mathcal P_2(2k)}\prod_{(r,s)\in\pi}\E[Y_{i_ri_{r+1}}^{\e_r}Y_{i_si_{s+1}}^{\e_s}]
\\=&\lim_{p\to \infty}\frac{1}{pn^k}\sum_{I_{2k}'}
\sum_{\pi\in  NC_2(2k)}\prod_{(r,s)\in\pi}\d_{i_ri_{s+1}}\d_{i_si_{r+1}}(1-\d_{\e_r\e_s}).
\end{align*}
Again $\d_{\e_r\e_s}=0$ when $(r,s)\in\pi\in NC_2(2k)$. Therefore we have
\begin{align}\label{eqn:21}
\lim_{p\to \infty}\frac{1}{p}\E[\Tr(\bar Y_p^k)]=\sum_{\pi\in NC_2(2k)}\lim_{n\to \infty}\frac{1}{p.n^k}\sum_{I_{2k}'}\prod_{r=1}^{2k}\d_{i_ri_{\gamma\pi(r)}}.
\end{align}
Comparing the right sides of \eqref{eqn:22} and \eqref{eqn:21} we get
\eqref{eqn:rect00}. This completes the proof.
\end{proof}

\begin{proof}[Proof of Theorem \ref{thm:assymrect}]
 Let $\bar Y_p^{(1)},\ldots, \bar Y_p^{(m)}$ be $m$ independent copies of $\bar Y_p$. Then $\bar Y_p^{(1)}, \ldots, \bar Y_p^{(m)} $ are asymptotically free (see \cite{capitaine}). Using this fact we show that $\bar X_p^{(1)},\ldots, \bar X_p^{(m)}$ are asymptotically free.
For simplicity we show that $\bar X_p^{(1)}$ and $\bar X_p^{(2)}$ are asymptotically free. It is enough to show that  
\begin{align*}
\lim_{p\to \infty}\vp_n(\bar X_p^{(\tau_1)}\cdots \bar X_p^{(\tau_k)})=\lim_{p\to \infty}\vp_n(\bar Y_p^{(\tau_1)}\cdots \bar Y_p^{(\tau_k)}),
\end{align*}
for $\tau_1,\tau_2,\ldots, \tau_k\in \{1,2\}$. We have
\begin{align*}
\vp_n(\bar X_p^{(\tau_1)}\cdots \bar X_p^{(\tau_k)})&=\frac{1}{p}\E \Tr[\bar X_p^{(\tau_1)}\cdots \bar X_p^{(\tau_k)}]
\\&=\frac{1}{pn^k}\E\Tr[X^{(\tau_1')\e_1}X^{(\tau_2')\e_2}\cdots X^{(\tau_{2k}')\e_{2k}}],
\end{align*}
where $\tau_{2m-1}'=\tau_{2m}'=\tau_m \in \{1,2\}$, $e_{2m-1}=1$ and $e_{2m}=*$ for $m=1,\ldots,k$. Therefore by the trace formula we get 
\begin{align*}
\frac{1}{pn^k}\E\Tr[X^{(\tau_1')\e_1}X^{(\tau_2')\e_2}\cdots X^{(\tau_{2k}')\e_{2k}}]&=\frac{1}{pn^k}\sum_{I_{2k}'}\E[X_{i_1i_2}^{(\tau_1')\e_1}\cdots X_{i_{2k}i_1}^{(\tau_{2k}')\e_{2k}}]
\\&= \frac{1}{pn^k}\sum_{I_{2k}'}\sum_{\pi\in \mathcal P_2(2k)}\prod_{(r,s)\in \pi}\E[X_{i_ri_{r+1}}^{(\tau_r')\e_r}X_{i_si_{s+1}}^{(\tau_s')\e_s}]+o(1),
\end{align*}
where $I_{2k}'=\{(i_1,\ldots, i_{2k})\; :\; 1\le i_{2m} \le n, 1\le i_{2m-1}\le p, m=1,\ldots, k\}$.
Using the same arguments  as in the proof of Theorem \ref{thm:rectangle} we get 
\begin{align}\label{eqn:aa'}
\lim_{p\to \infty}\vp_n(\bar X_p^{(\tau_1)}\cdots \bar X_p^{(\tau_k)})=\sum_{\pi\in NC_2(2k)}\lim_{p\to \infty}\frac{1}{pn^k}\sum_{I_{2k}'} \prod_{(r,s)\in \pi}\d_{i_ri_{s+1}}\d_{i_si_{r+1}}\d_{\tau_r'\tau_s'}.
\end{align} 
 On the other hand we have 
\begin{align*}
\vp_n(\bar Y_p^{(\tau_1)}\cdots \bar Y_p^{(\tau_k)})&=\frac{1}{p}\E \Tr[\bar Y_p^{(\tau_1)}\cdots \bar Y_p^{(\tau_k)}]
=\frac{1}{pn^k}\E\tr[Y^{(\tau_1')\e_1}Y^{(\tau_2')\e_2}\cdots Y^{(\tau_{2k}')\e_{2k}}]
\\&=\frac{1}{pn^k}\sum_{I_{2k}'}\E[Y_{i_1i_2}^{(\tau_1')\e_1}\cdots Y_{i_{2k}i_1}^{(\tau_{2k}')\e_{2k}}]
\\&= \frac{1}{pn^k}\sum_{I_{2k}'}\sum_{\pi\in \mathcal P_2(2k)}\prod_{(r,s)\in \pi}\E[Y_{i_ri_{r+1}}^{(\tau_r')\e_r}Y_{i_si_{s+1}}^{(\tau_s')\e_s}].
\end{align*}
Again, as the elements are i.i.d. $N(0,1)$, we have 
\begin{align*}
\E[Y_{i_ri_{r+1}}^{(\tau_r')\e_r}Y_{i_si_{s+1}}^{(\tau_s')\e_s}]=(\d_{i_ri_s}\d_{i_{r+1}i_{s+1}}
\d_{\e_r\e_s}+\d_{i_ri_{s+1}}\d_{i_si_{r+1}}(1-\d_{\e_r\e_s}))\d_{\tau_r'\tau_s'}.
\end{align*}
Using Lemma \ref{lem:wigner} and the arguments as in the proof of Theorem \ref{thm:rectangle} we get 
\begin{align}\label{eqn:bb'}
\lim_{p\to \infty}\vp_n(\bar Y_p^{(\tau_1)}\cdots \bar Y_p^{(\tau_k)})=\sum_{\pi\in NC_2(2k)}\lim_{p\to \infty}\frac{1}{pn^k}\sum_{I_{2k}'} \prod_{(r,s)\in \pi}\d_{i_ri_{s+1}}\d_{i_si_{r+1}}\d_{\tau_r'\tau_s'}.
\end{align}
Therefore from \eqref{eqn:aa'} and \eqref{eqn:bb'}, for  $\tau_1,\ldots,\tau_k\in \{1,2\}$, we have 
\begin{align*}
\lim_{p\to \infty}\vp_n(\bar X_p^{(\tau_1)}\cdots \bar X_p^{(\tau_k)})=\lim_{p\to \infty}\vp_n(\bar Y_p^{(\tau_1)}\cdots \bar Y_p^{(\tau_k)}).
\end{align*}
Therefore $(\bar X_p^{(1)}, \bar X_{p}^{(2)})$ and $(\bar Y_p^{(1)}, \bar Y_{p}^{(2)})$ jointly converge to the same limit. So $\bar X_p^{(1)}$ and $\bar X_p^{(2)}$ are asymptotically free, as $\bar Y_p^{(1)}$ and $\bar Y_p^{(2)}$ are asymptotically free.  Using the same arguments it can be shown that $\bar X_p^{(1)},\ldots, \bar X_{p}^{(k)}$ are asymptotically free.
\end{proof}

\section{Proof of Theorem \ref{thm:brown}}\label{sec:brown}

In this section we give the proof of Theorem \ref{thm:brown}. The following lemma links products of circular elements and of elliptic elements. Its proof is given later. 

\begin{lemma}\label{lem:distribution}
Let $e_1,\ldots, e_k$ be $k$ free elliptic elements with parameters $0\le\rho_1,\ldots, \rho_k\le 1$ respectively and $c_1,\ldots,c_k$ be $k$ free circular elements. 
Then $e_1\cdots e_k e_k^*\cdots e_1^*$ and $c_1\cdots c_kc_k^*\cdots c_1^*$ have the same distribution.
\end{lemma}

Two well known results will be used in the proof of Theorem \ref{thm:brown}. The first result gives the $S$-transform of $cc^*$ where $c$ is a circular element. The {\it $S$-transform} of $a\in \mathcal A$ is defined by $S_a(z):=\frac{1}{z}R_a^{<-1>}(z)$, where $R_a(z)=\sum_{n=1}^{\infty}\kappa_n(a,\ldots, a)z^n$ and $f^{<-1>}$ denotes the  inverse of $f$ under the composition mapping (see \cite{speicherbook}, page 294). 

\begin{result}\label{ft:brown}
Let $c$ be a standard circular element in an NCP $(\mathcal A, \vp)$. Then the $S$ transform of $cc^*$ is given by
\begin{align*}
S_{cc^*}(z)=\frac{1}{1+z}.
\end{align*}
\end{result}

The second result gives the Brown measure of any $R$-diagonal element. An element $a\in \mathcal A$ is called {\it $R$-diagonal} if $\kappa_n(a_1,\ldots, a_n)=0$ for all $n\in \N$ whenever the arguments $a_1,\ldots, a_n\in \{a,a^*\}$ are not alternating in $a$ and $a^*$.  For more details on $R$-diagonal elements we refer to Lecture 15 in \cite{speicherbook}.

\begin{result}\label{res:brownR}
Suppose $x$ is an $R$-diagonal element. Then its Brown measure $\mu_x$ is rotationally invariant and can be described by the probabilities  
\begin{align*}
\mu_x(\{\lambda\in \C \;:\; |\lambda|\le t\})=\left\{\begin{array}{lcr}
0 & \mbox{for} & t\le \frac{1}{\sqrt{\vp(xx^*)^{-1}}}\\
1+S_{xx^*}^{<-1>}(t^{-2}) & \mbox{for} & \frac{1}{\sqrt{\vp(xx^*)^{-1}}}\le t \le \sqrt{\vp(xx^*)}\\
1 & \mbox{for } & t\ge \sqrt{\vp(xx^*)},
\end{array} \right.
\end{align*}
where $f^{<-1>}$ denotes the  inverse of $f$ under the composition mapping.
\end{result}

For proofs of Results \ref{ft:brown} and \ref{res:brownR}, we refer the reader to  \cite{belinschi}, \cite{haagerup}.

\begin{proof}[Proof of Theorem \ref{thm:brown}]
It is known that the product of free elliptic elements is $R$-diagonal (see \cite{soshnikov}, p. 9). Therefore $e_1\cdots e_k$ is an $R$-diagonal element for $k\ge 2$ and hence its Brown measure is determined by the distribution of $e_1\cdots e_ke_k^*\cdots e_1^*$. By Lemma \ref{lem:distribution}, this distribution is same as the distribution of $c_1\cdots c_kc_k^*\cdots c_1^*$. Let $S_k$ denote the S-transform of $c_1\cdots c_kc_k^*\cdots c_1^*$. By Result \ref{ft:brown} and freeness we have  
\begin{align*}
S_k(z)=S_{c_1c_1^*}(z)\cdots S_{c_kc_k^*}(z)=\l(\frac{1}{1+z}\r)^k,
\end{align*}
Therefore, for $t>0$, we get 
\begin{align}\label{eqn:brown}
S_k^{<-1>}(t)=1-t^{-\frac{1}{k}}.
\end{align} 
Clearly, $\vp(c_1\cdots c_kc_k^*\cdots c_1^*)=1$ and $\vp((c_1\cdots c_k c_k^*\cdots c_1^*)^{-1})=\infty$, as $c_ic_i^*$ has the quarter-circle distribution. Therefore using \eqref{eqn:brown} in Result \ref{res:brownR} we have 
\begin{align*}
\mu_k(\{z\;:\; |z|\le t\})=\l\{\begin{array}{ll}
t^{\frac{2}{k}} & \mbox{for $t\le 1$}\\
1 & \mbox{for $t\ge 1$}.
\end{array} \r.
\end{align*}
Hence the result.
\end{proof}

It remains to prove Lemma \ref{lem:distribution}.

\begin{proof}[Proof of Lemma \ref{lem:distribution}]
Since $e_1\cdots e_ke_k^*\cdots e_1^*$ and $c_1\cdots c_kc_k^*\cdots c_1^*$ are self-adjoint elements, it is enough to show that all their moments agree. From the proof of Lemma \ref{lem:moments of e}, we have $\kappa_2(e,e)=\kappa_2(e^*,e^*)=\rho$ and $\kappa_2(e,e^*)=\kappa_2(e^*,e)=1$. It is  easy to see that $\kappa_2(c,c)=\kappa_2(c^*,c^*)=0$ and $\kappa_2(c,c^*)=\kappa_2(c^*,c)=1$. We write 
\begin{align*}
(e_1\cdots e_ke_k^*\cdots e_1^*)^n=e_{\tau_1}^{\e_1}e_{\tau_2}^{\e_2}\cdots e_{\tau_{2kn}}^{\e_{2kn}},
\end{align*}
where $\tau_{2mk+j}=\tau_{2(m+1)k-j+1}=j$ , $e_{2mk+j}=1$ and  $e_{(2m+1)k+j}=*$ for $j=1,\ldots,k$ and $m=0,\ldots,n-1$.

For positive integer $n$, applying the moment-free cumulant formula, as the higher order free cumulants are zero, we have
\begin{align*}
\vp((e_1\cdots e_ke_k^*\cdots e_1^*)^n)=\sum_{\pi\in NC_{2}(2nk)}\prod_{(r,s)\in \pi}\kappa_2(e_{\tau_r}^{\epsilon_r},e_{\tau_s}^{\epsilon_s}),
\end{align*}
where $\tau_r,\tau_s\in \{1,\ldots, k\}$ and $\e_r, \e_s\in \{1,*\} $. Observe that if $(r,s)\in \pi\in NC_2(2nk)$, then one of $\{r,s\}$ is odd and the other is even. Further, if $e_{\tau_r}$ appears at an odd place then $e_{\tau_r}^*$ appears only at an even place and vice-versa. Indeed, consider the word $e_1\cdots e_k e_k^*\cdots e_1^*$. Then $e_r$ appears at the $r$-th place and $e_r^*$ appears at the $(2k+1-r)$-th place. Clearly if $e_r$ appears at an odd place then $e_r^*$ appears at an even place and vice-versa. Therefore $\kappa_2(e_{\tau_r}^{\epsilon_r},e_{\tau_s}^{\epsilon_s})$ will be of the form either $\kappa_2(e,e^*)$ or $\kappa_2(e^*,e)$ and in both cases equal to $1$. Thus, as the mixed free cumulants are same for circular and elliptic elements, we have 
\begin{align*}
\vp((e_1\cdots e_ke_k^*\cdots e_1^*)^n)=\sum_{\pi\in NC_{2}(2nk)}\prod_{(r,s)\in \pi}\kappa_2(c_{\tau_r}^{\epsilon_r},c_{\tau_s}^{\epsilon_s})=\vp((c_1\cdots c_kc_k^*\cdots c_1^*)^n).
\end{align*}
This completes the proof.
\end{proof}

\bibliography{bibtex}

\providecommand{\bysame}{\leavevmode\hbox to3em{\hrulefill}\thinspace}
\providecommand{\MR}{\relax\ifhmode\unskip\space\fi MR }
\providecommand{\MRhref}[2]{%
  \href{http://www.ams.org/mathscinet-getitem?mr=#1}{#2}
}
\providecommand{\href}[2]{#2}
\begin{thebibliography}{10}

\bibitem{baisilversteinbook}
Zhidong Bai and Jack~W. Silverstein, \emph{Spectral analysis of large
  dimensional random matrices}, second ed., Springer Series in Statistics,
  Springer, New York, 2010. \MR{2567175}

\bibitem{belinschi}
Serban Belinschi, Piotr Sniady, and Roland Speicher, \emph{Eigenvalues of
  non-hermitian random matrices and brown measure of non-normal operators:
  hermitian reduction and linearization method}, arXiv preprint
  arXiv:1506.02017 (2015).

\bibitem{bosehazra}
Arup Bose, Rajat Subhra~Hazra, and Koushik Saha, \emph{Patterned random
  matrices and method of moments}, Proceedings of the {I}nternational
  {C}ongress of {M}athematicians. {V}olume {IV}, Hindustan Book Agency, New
  Delhi, 2010, pp.~2203--2231. \MR{2827968}

\bibitem{brown}
L.~G. Brown, \emph{Lidskii 's theorem in the type {${\rm II}$} case}, Geometric
  methods in operator algebras ({K}yoto, 1983), Pitman Res. Notes Math. Ser.,
  vol. 123, Longman Sci. Tech., Harlow, 1986, pp.~1--35. \MR{866489}

\bibitem{capitaine}
M.~Capitaine and M.~Casalis, \emph{Asymptotic freeness by generalized moments
  for {G}aussian and {W}ishart matrices. {A}pplication to beta random
  matrices}, Indiana Univ. Math. J. \textbf{53} (2004), no.~2, 397--431.
  \MR{2060040}

\bibitem{dykema}
Ken Dykema, \emph{On certain free product factors via an extended matrix
  model}, J. Funct. Anal. \textbf{112} (1993), no.~1, 31--60. \MR{1207936}

\bibitem{fuglede}
Bent Fuglede and Richard~V. Kadison, \emph{Determinant theory in finite
  factors}, Ann. of Math. (2) \textbf{55} (1952), 520--530. \MR{0052696}

\bibitem{naumov}
F.~G\"otze, A.~Naumov, and A.~Tikhomirov, \emph{On minimal singular values of
  random matrices with correlated entries}, Random Matrices Theory Appl.
  \textbf{4} (2015), no.~2, 1550006, 30. \MR{3356884}

\bibitem{manjusinglering}
Alice Guionnet, Manjunath Krishnapur, and Ofer Zeitouni, \emph{The single ring
  theorem}, Ann. of Math. (2) \textbf{174} (2011), no.~2, 1189--1217.
  \MR{2831116}

\bibitem{sommer}
H.~Sompolinsky H.~Sommers, A.~Crisanti and Y.~Stein., \emph{Spectrum of large
  random asymmetric matrices}, Phys. Rev. Lett. \textbf{60} (1988).

\bibitem{haagerup}
Uffe Haagerup and Flemming Larsen, \emph{Brown's spectral distribution measure
  for {$R$}-diagonal elements in finite von {N}eumann algebras}, J. Funct.
  Anal. \textbf{176} (2000), no.~2, 331--367. \MR{1784419}

\bibitem{hiai}
Fumio Hiai and D\'enes Petz, \emph{Asymptotic freeness almost everywhere for
  random matrices}, Acta Sci. Math. (Szeged) \textbf{66} (2000), no.~3-4,
  809--834. \MR{1804226}

\bibitem{larsen}
F.~Larsen, \emph{Brown measures and r-diagonal elements in finite von neumann
  algebras}, PhD Thesis, Department of Mathematics and Computer Science,
  University of Southern Denmark (1999).

\bibitem{pastur}
V.~A. Mar\v{c}enko and L.~A. Pastur, \emph{Distribution of eigenvalues in
  certain sets of random matrices}, Mat. Sb. (N.S.) \textbf{72 (114)} (1967),
  507--536. \MR{0208649}

\bibitem{nguyen}
Hoi~H. Nguyen and Sean O'Rourke, \emph{The elliptic law}, Int. Math. Res. Not.
  IMRN (2015), no.~17, 7620--7689. \MR{3403996}

\bibitem{speicherbook}
Alexandru Nica and Roland Speicher, \emph{Lectures on the combinatorics of free
  probability}, London Mathematical Society Lecture Note Series, vol. 335,
  Cambridge University Press, Cambridge, 2006. \MR{2266879}

\bibitem{soshnikov}
Sean O'Rourke, David Renfrew, Alexander Soshnikov, and Van Vu, \emph{Products
  of independent elliptic random matrices}, J. Stat. Phys. \textbf{160} (2015),
  no.~1, 89--119. \MR{3357969}

\bibitem{dimitri}
Dimitri Shlyakhtenko, \emph{Limit distributions of matrices with bosonic and
  fermionic entries}, Free probability theory ({W}aterloo, {ON}, 1995), Fields
  Inst. Commun., vol.~12, Amer. Math. Soc., Providence, RI, 1997, pp.~241--252.
  \MR{1426842}

\bibitem{sniady}
Piotr \'Sniady, \emph{Random regularization of {B}rown spectral measure}, J.
  Funct. Anal. \textbf{193} (2002), no.~2, 291--313. \MR{1929504}

\bibitem{speicher}
Roland Speicher, \emph{Free convolution and the random sum of matrices}, Publ.
  Res. Inst. Math. Sci. \textbf{29} (1993), no.~5, 731--744. \MR{1245015}

\bibitem{voiculescu1}
Dan Voiculescu, \emph{Limit laws for random matrices and free products},
  Invent. Math. \textbf{104} (1991), no.~1, 201--220. \MR{1094052}

\bibitem{voiculescu2}
\bysame, \emph{A strengthened asymptotic freeness result for random matrices
  with applications to free entropy}, Internat. Math. Res. Notices (1998),
  no.~1, 41--63. \MR{1601878}

\bibitem{wachter}
Kenneth~W. Wachter, \emph{The strong limits of random matrix spectra for sample
  matrices of independent elements}, Ann. Probability \textbf{6} (1978), no.~1,
  1--18. \MR{0467894}

\bibitem{yin}
Y.~Q. Yin, \emph{Limiting spectral distribution for a class of random
  matrices}, J. Multivariate Anal. \textbf{20} (1986), no.~1, 50--68.
  \MR{862241}

\end{thebibliography}

\bibliographystyle{amsplain}

\end{document}